\documentclass[11pt,reqno]{amsart}

\usepackage{amsmath, amsthm, amssymb}
\usepackage{amsfonts}
\usepackage[ansinew]{inputenc}
\usepackage[dvips]{epsfig}
\usepackage{graphicx}
\usepackage[english]{babel}
\usepackage{hyperref}

\usepackage{cite}
\usepackage{graphicx}
\usepackage{amscd}
\usepackage{xcolor}
\usepackage{bm}
\usepackage{enumerate}

\usepackage{verbatim}
\usepackage{hyperref}
\usepackage{amstext}
\usepackage{latexsym}
\theoremstyle{plain}
\newtheorem{theorem}{Theorem}[section]
\newtheorem{definition}[theorem]{Definition}

\newtheorem{proposition}[theorem]{Proposition}
\newtheorem{lemma}[theorem]{Lemma}
\newtheorem{remark}[theorem]{Remark}

\numberwithin{theorem}{section}
\numberwithin{equation}{section}

\newcommand{\average}{{\mathchoice {\kern1ex\vcenter{\hrule height.4pt
width 6pt depth0pt} \kern-9.7pt} {\kern1ex\vcenter{\hrule
height.4pt width 4.3pt depth0pt} \kern-7pt} {} {} }}

\def\R{\mathbb{R}}

\newcommand{\D }{\Delta }

\newcommand{\G }{\Gamma}

\renewcommand{\O }{\Omega }
\newcommand{\Om}{{\Omega}}

\newcommand{\be}{\begin{equation}}
\newcommand{\ee}{\end{equation}}

\renewcommand{\k}{\kappa}

 \usepackage{mathrsfs}

\newcommand{\dist}{{\rm dist}}
\newcommand{\supp}{{\rm supp}}

\DeclareMathOperator{\id}{id}

\renewcommand{\epsilon}{\varepsilon}

\newcommand{\Ds}{ (-\D)^s}

\begin{document}
\title{Shape derivative approach to fractional overdetermined problems}

 \author[]
{Sidy Moctar  Djitte${}^{1}$, Ignace Aristide Minlend ${}^2$}


\address{${}^1$D\'epartement de Math\'ematiques, UFR Sciences Exactes et Naturelles, Moulin de la Housse, BP 1039, 51687 REIMS Cedex, France.}
\email{sidy-moctar.djitte@univ-reims.fr, sidy.m.djitte@aims-senegal.org}
\address{${}^2$ Faculty of Economics and Applied Management, University of Douala,  BP 2701, Douala, Littoral Province, Cameroon.}
\email{\small{ignace.a.minlend@aims-senegal.org} }

\begin{abstract}
\noindent 
We use shape derivative approach  to prove that balls are the only convex and  $C^{1,1}$ regular domains in which the fractional overdetermined problem
\begin{equation*}
\left\{\begin{aligned}
\Ds u&= \lambda_{s, p} u^{p-1}\quad\text{in}\quad\Om \\
u &= 0\quad \text{in}\quad\R^N\setminus \Om\\
u/\delta^s&=C_0\quad\text{on\;\; $\partial\O$}
\end{aligned}
\right. 
\end{equation*}
admits a nontrivial solution for $p\in [1, 2]$ and where  $\lambda_{s, p}= \lambda_{s, p}(\O)$ is the best constant in the family of Subcritical Sobolev inequalities. 
In the cases $p=1$ and $p=2$, 
we recover the classical symmetry results of Serrin, corresponding to the torsion problem and the first Dirichlet eigenvalue problem, respectively (see \cite{FS-15}). We note that for $p\in (1,2)$, the above problem lies outside the framework of \cite{FS-15}, and the methods developed therein do not apply.  
Our approach extends to the fractional setting a method initially developed by A. Henrot and T. Chatelain in \cite{CH-99}, and relies on the use of domain derivatives combined with the continuous Steiner symmetrization introduced by Brock in \cite{Brock-00}.
\end{abstract}
\maketitle


\setcounter{equation}{0}
\section{Introduction}
Let $s\in (0,1)$ and  $\O$ be a bounded open set in $\R^N$ with $N\geq 1$.  The present paper is devoted to  a rigidity result  for the overdetermined problem 
\begin{equation}\label{ov-det12}
\left\{\begin{aligned}
\Ds u&= \lambda_{s, p} u^{p-1}\quad\text{in}\quad\Om \\
u &= 0\quad \text{in}\quad\R^N\setminus \Om\\
u/\delta^s&=C_0\quad\text{on\;\; $\partial\O$}
\end{aligned}
\right. 
\end{equation}
where $p\in [1,2]$, $\delta(\cdot):=\textrm{dist}(\cdot, \partial \O)$  and $(-\Delta)^s$ stands for the fractional laplacian of order $s$ defined for sufficiently regular functions, say, $\varphi  \in C^{1, 1}_c(\R^N)$ by the principal value integral
\[
(-\D)^s\varphi(x) = c_{N,s}p.v\int_{\R^N}\frac{\varphi(x)-\varphi(y)}{|x-y|^{N+2s}}dy. 
\]
with $c_{N,s}=\pi^{-N/2}s4^s\frac{\G(N/2+s)}{\G(1-s)}$.  The reason for the
 presence of the normalization constant $c_{N,s}$ is to match with another natural definition, based on the Fourier transform which sets the fractional Laplace operator $(-\Delta)^s$ as the Fourier multiplier of symbol $|\xi|^{2s}$. In  \eqref{ov-det12}, $u/\delta^s$  is defined on $\partial \O$ as the limit 
 \begin{equation}\label{fractnormaderi}
 \frac{u}{\delta^s}(x_0):=\lim_{x\rightarrow x_0, x\in \O} \frac{u}{\delta^s}(x)
 \end{equation}
provided it exists.  Weak solutions to the first two equations in \eqref{ov-det12} corresponds to functions for which the best constant in the inequality  
\begin{equation*}\label{Sob-best}
C\|u\|^2_{L^p(\O)}\leq  [u]_{H^s(\R^N)}^2, \quad \textrm{ for all}\quad  u\in \mathcal{H}^s_0(\Om),
\end{equation*}
is attained. That is, functions that realizes the minimum in 
\begin{equation}\label{eqlamsp}
\lambda_{s, p}(\O):=\inf\Bigl\{ [u]_{H^s(\R^N)}^2: u\in \mathcal{H}^s_0(\Om),\quad  \|u\|^2_{L^p(\O)}=1\Bigl\},
\end{equation} 
when $p$ is subcritical, i.e, $p\in [1, \frac{2N}{N-2s})$ if $2s<N$ and $p\in [1, \infty)$ if $2s\leq N=1$. 
Here $\mathcal{H}^s_0(\Om)$ is the completion of $C^{0,1}_c(\Om)$ with respect to the Gagliardo semi-norm defined by 
$$
[u]_{H^s(\R^N)}^2:= \frac{c_{N,s}}{2}\iint_{\R^N\times\R^N}\frac{(u(x)-u(y))^2}{|x-y|^{N+2s}}dxdy.
$$
That the infimum in \eqref{eqlamsp} is attained follows from standard arguments in the calculus of variations. For $p\in [1,2]$, it is well known that the solution to the first two equations --or the function that achieves the infimum in \eqref{eqlamsp}-- is unique (see \cite[Lemma A.1]{DFW-21}). Moreover, by standard regularity theory (see, e.g., \cite{Sylves, RS, PalaSavinVal}), the solution satisfies  $u\in C^s(\overline{\O})\cap C^\infty(\O)$. 
If, in addition,  $\O$ is of class $C^{1,1}$, then $u/\delta^s\in C^0(\overline{\O})$. Therefore, the overdetermined condition in \eqref{ov-det12} and the whole equation is well-defined in the  pointwise sense.
\\
In the seminal work \cite{Serrin}, Serrin proved, in the case $s=1$ and $p=1$, that if \eqref{ov-det12} admits a nontrivial solution, then $\O$ must be a round ball. Serrin's work initiated extensive research in the field of overdetermined problems, and numerous extensions and alternative approaches to his result have since been proposed see for instance \cite{BH-02, CH-99 , Wb, SS} and the references therein.\\\\
In \cite{FS-15}, Fall and Jarohs extends Serrin's result to the case $s\in (0,1)$, proving that if $f:\R\to\R$ is locally Lipschitz continuous and the equation \begin{equation}\label{f-ov-det}
\left\{\begin{aligned}
\Ds u&= f(u)\quad\text{in}\quad\Om \\
u &= 0\quad \text{in}\quad\R^N\setminus \Om\\
u&>0\quad\text{in $\O$}\\
u/\delta^s&=C_0\quad\text{on\;\; $\partial\O$}
\end{aligned}
\right. 
\end{equation}
has a nontrivial solution, and if moreover 
\begin{equation}\label{c1-extension}
u/\delta^s\quad \textrm{has a $C^1$ extension up to the boundary}\footnote{In \cite{FS-15}, the assumption \eqref{c1-extension} was missing. The necessity of this assumption for the result stated there was pointed out to us by Sven Jarohs; see, for example \cite[remark 3.1]{FSvenTadeuszSalani}.}   
\end{equation}
then $\O$ must be a ball. The assumption \eqref{c1-extension}, is satisfied if $\Omega$ is of class $C^{2,\gamma}$ (see e.g \cite{RS-Fn}). We stress that the assumption \eqref{c1-extension} is reasonable since even in the classical case $s=1$ considered by Serrin, the solution is required to be $C^2(\overline{\Omega})$. 

An early rigidity result for equation \eqref{f-ov-det} was first obtained in  \cite{DV} where the authors considered the particular case $f=1$, $N=2$ and $s=1/2$. The arguments in the aforementioned papers is based in the moving plane method and heavily relies on the Lipschtiz continuity of the nonlinearity $f$. In this work, we consider problem \eqref{f-ov-det} with a pure power nonlinearity $f(t) = \lambda t^{p-1}$ where $p\in (1,2)$. Note that this nonlinearity does not fall within the framework of \cite{FS-15}, as it is not locally Lipschitz continuous in $[0,\infty)$. Nevertheless, we show that if $\O$ is convex and of class $C^{1,1}$, the same symmetry result holds. Our approach do not involve the regularity assumptions in \cite{FS-15}  and builds on  a  method pioneered by A. Henrot and T. Chatelain, and is based on the use of shape derivatives. \\\\
Here and in the following, we call a $C^{1,1}$ bounded open set $\O$ a solution of \eqref{ov-det12} if the later admits a nontrivial point-wise solution $u$. Our result can be stated as follows.
\begin{theorem}\label{Theorem11}Let $p\in [1,2]$. Then, every convex solution $\Om$ of \eqref{ov-det12} must be a ball.
\end{theorem}

This result is closely related to the work of Fall and Jarohs \cite{FS-15}, who treated the case $p\in \{1\}\cup  [2, \infty)$ without assuming convexity, but under a slightly stronger regularity condition that $\O$ is of class $C^2$. By contrast, our result requires only $C^{1,1}$ regularity, though it crucially depends on the convexity of $\O$. We also note that our method uses the differentiability of the shape functional $\O\mapsto J_{s,p}(\O)$, defined in \eqref{eqenerJ} and this requirement restrict our analysis to the range $p\in [1,2]$. 

\par\;

The method presented here may not address the most general case, but it is of independent interest as it offers an alternative to the moving plane method. Unlike the moving plane approach, which does not generalize well when, for instance, the nonlinearity in \eqref{ov-det12} is not Lipschitz continuous or when the underlying operator is nonlinear, the method presented here remains applicable (see, e.g., \cite{BH-02}). 

 \par\;
 
The paper is organized as follows: Section \ref{sec-2} contains some preliminaries results on continuous Steiner symmetrization.  In section \ref{sec-3}, we relate solution domains $\O$  to  the overdetermined problem  \eqref{ov-det12} with critical sets of some specific domain dependent functional and in Section \ref{sec-4}, we give the proof of Theorem \ref{Theorem11}.

\section{Preliminaries on continuous  Steiner symmetrization}\label{sec-2}

In this section, we recall the definition of continuous Steiner symmetrization of sets and continuous functions in $\R^N$, and state some basic properties that will be used later in this note. Several definitions of continuous Steiner symmetrization can be found in the literature, differing primarily in the speed at which intervals are translated. Here, we adopt the definition introduced in \cite{Brock-00}, to which we refer for detailed proofs. For alternative formulations, the interested reader is referred to the relevant literature.

\begin{definition}\label{defn:cont set}Let $\mathcal{M}(\R)$ be the set of Lebesgue measurable sets in $\R$. We call a continuous Steiner symmetrization, a family of set transformations 
\[
E_t: \mathcal{M}(\R)\to \mathcal{M}(\R), \quad 0\leq t\leq +\infty,
\]
satisfying the properties $(M, N\subset \mathcal{M}(\R), 0\leq s,\,t\leq +\infty)$:
\begin{enumerate}
\item $E_t(M)= |M|$, \qquad\textrm{(equimeasurability)}
\item If $M\subset N$, then $E_t(M)\subset E_t(N)$,\qquad\textrm{(monotonicity)}
\item $E_t(E_s(M)) = E_{t+s}(M)$, \qquad(\textrm{semigroup properties)}
\item If $M = [a,b]$ is a bounded closed interval, then $E_t(M) = [a^t,b^t]$ where
\begin{align}\label{a-t}
a^t &= (1/2)\Big(a-b+e^{-t}(a+b)\Big),\\
b^t &= (1/2)\Big(b-a+e^{-t}(b+a)\Big).\label{b-t}
\end{align} 

\end{enumerate}
\end{definition}
We note that the existence of such transformation is guaranteed by \cite[Theorem 2.1]{Brock-00}.
Next we define the continuous Steiner symmetrization of open set in $\R^N$, with $N\geq 2$. Let $\Om\subset\R^N$ be an open set in $\R^N$. We denote by $\O'$ the projection of $\O$ onto the hyperplane $\{x_N = 0\}$. That is,
\[
\O' = \Big\{ x'\in \R^{N-1}: \exists\;\; y\in \R\quad \textrm{and}\quad (x',y)\in \O)\Big\}.
\]
For a fixed $x'\in \O'$, we define the vertical section of $\O$ at $x'$
  by
\[
\O(x'):=\Big\{ y\in \R: \; (x',y)\in \O\Big\}.
\]
\begin{definition}
Let $t\in [0,\infty]$ and $\O\subset  \R^N$ be an open set. The continuous Steiner symmetrization of $\O$ with respect to the hyperplane $\{x_N = 0\}$ denoted by $\O^t$ is defined by
\begin{equation}\label{Om-t}
 \O^t = \Big\{ (x',y): \; x'\in \Om'\;\; \textrm{and}\;\; y\in (\O(x'))^t\Big\}.
\end{equation}
\end{definition}
By \cite[Lemma 2.1]{Brock-00}, if $\O$ is open, then the symmetrization $\O^t$ has an open representative which we still denote by $\O^t$. From now, if one speak of continuous Steiner symmetrization of open set, we always mean its open representative.  Finally, if $u\in C^0(\R^N)\cap L^1(\R^N)$ is a continuous, nonnegative function, we define its superlevel sets for  $h>0$ by
\[
\O(h) := \Big\{ x\in \R^N: u(x)>h\Big\}.
\]
\begin{definition}\label{Steiner-sym-u}
For any $t\in [0,+\infty]$, the continuous Steiner symmetrization of $u$ with respect to the hyperplane $\R^{N-1}$ denoted by $u^t: \R^N\to \R$ is defined by
\begin{equation}\label{def-C-St-sym}
u^t(x) = \int_{0}^\infty \chi_{\O((h))^t}(x)dh.
\end{equation}   
\end{definition}
As a consequence of the definition, it follows from the layer-cake representation that the continuous Steiner symmetrization preserves the $L^p$ norm for 
all $p\in[1,\infty]$. Moreover the following properties hold true (see \cite[Lemma $3$ and Theorem $7$]{Brock-95})
\begin{lemma}\label{non-exp}
Let $ p\in [1, \infty]$ and $u, v\in L^p(\R^N)$ be two nonnegative continuous functions. Then for all $t\in [0,\infty]$, there holds:
\[
\big\|u^t-v^t\big\|_{L^p(\R^N)}\leq \big\|u-v\big\|_{L^p(\R^N)}.
\]
Moreover, if $u\in C^{0,1}(\R^N)$, then $u^t\in C^{0,1}(\R^N)$ with the same Lipchitz constant. 
\end{lemma}

The following result state that if $f:\R^N\to\R$ is not radially decreasing about any center, then the function $t\mapsto [f^t]^2_{H^s(\R^N)}$ is a strictly decreasing provided that $t\geq 0$ is small enough.
\begin{theorem}\cite[Theorem 1]{Delgadino}\label{thm-del}
Let $u\in L^1(\R^N)\cap C^0(\R^N)$ and $u^t$ be the continuous Steiner symmetrization of $u$ according the definition \eqref{Steiner-sym-u}. Assume $u$ is not radially decreasing about any center. Then there are constants $\gamma_0
 =\gamma_0(N,s,u) > 0$,
$t_0 = t_0(u)>0$, and a hyperplane $H$ such that
 \begin{equation}\label{main-prop-0}[u^t]^2_{H^s(\R^N)}\leq [u]^2_{H^s(\R^N)}-\gamma_0 t\quad \textrm{for all \;\;$0\leq t\leq t_0$}
\end{equation}
where
\[
[u]^2_{H^s(\R^N)}:=\frac{c_{N,s}}{2}\iint_{\R^N\times\R^N}\frac{(u(x)-u(y))^2}{|x-y|^{N+2s}}dxdy.
\]
and $u^t$ is the continuous Steiner symmetrization of $u$ with respect to the hyperplane $H$. Moreover, the hyperplane $H$ may be taken to be $\{x_N = 0\}$.
\end{theorem}
\begin{remark}
Theorem \ref{thm-del} was originally proved in \cite{Delgadino} using a slightly different definition of continuous Steiner symmetrization. However, it is not difficult to verify that the same argument remains valid under the definition of continuous Steiner symmetrization adopted in the present work.
\end{remark}
As a corollary of Theorem \ref{thm-del}, we prove:
\begin{proposition}\label{prop-1} 

Let \( \Omega \subset \mathbb{R}^N \) be a bounded open set, and let 
$
u \in \mathcal{H}^s_0(\Omega) \cap C^0(\mathbb{R}^N)
$
be a nonnegative, continuous function. For each \( t \in [0, \infty] \), let \( \Omega^t \) be the continuous Steiner symmetrization of $\O$ defined as in equation~\eqref{Om-t}. Then for all \( t \in [0,\infty] \), the symmetrized function \( u^t \) belongs to \( \mathcal{H}^s_0(\Omega^t) \) and 
\begin{equation}\label{main-prop-0}
\|u^t\|^2_{\mathcal{H}^s_0(\O^t)}\leq \| u\|
^2_{\mathcal{H}^s_0(\O)}.
\end{equation}
If moreover $u$ is not radially decreasing about any point in \( \mathbb{R}^N \), then there exist constants 
$
\gamma_0 = \gamma_0(N, s, u) > 0$, $t_0 = t_0(u) > 0$ such that for all \( 0 < t \leq t_0 \), there holds
 \begin{equation}\label{main-prop-1}
\|u^t\|^2_{\mathcal{H}^s_0(\O^t)}\leq \| u\|
^2_{\mathcal{H}^s_0(\O)}-\gamma_0 t.
\end{equation}
\end{proposition}

\begin{proof}
We first show that 
\begin{equation}\label{pr}
[u^t]^2_{H^s(\R^N)}\leq [u]^2_{H^s(\R^N)}
\end{equation}
for all $t\in [0,\infty]$ and for all nonnegative function $u$. The proof follows a standard argument, we add it here for sake of completeness. Let $\varepsilon>0$ and consider the regularized kernel $\k_\varepsilon(z) = (|z|^2+\varepsilon^2)^{-(N+2s)/2}$. Since $\k_\varepsilon\in L^1(\R^N)$, and the $L^2$-norms of the symmetrized functions are preserved, we may write
\begin{align}\label{Eq2.8}
\frac{2}{c_{N,s}}J_\varepsilon[u^t] &:= \iint_{\R^N\times\R^N}\frac{(u^t(x)-u^t(y))^2}{\big(|x-y|^2+\varepsilon^2\big)^{\frac{N+2s}{2}}}dxdy\nonumber    \\
&=\iint_{\R^N\times\R^N}\frac{(u^t(x))^2+(u^t(y))^2-2u^t(x)u^t(y)}{\big(|x-y|^2+\varepsilon^2\big)^{\frac{N+2s}{2}}}dxdy\nonumber \\
&=\iint_{\R^N\times\R^N}\frac{u^2(x)}{\big(|x-y|^2+\varepsilon^2\big)^{\frac{N+2s}{2}}}dxdy+\iint_{\R^N\times\R^N}\frac{u^2(x)}{\big(|x-y|^2+\varepsilon^2\big)^{\frac{N+2s}{2}}}dxdy\nonumber\\
&\quad-2\iint_{\R^N\times\R^N}\frac{u^t(x)u^t(y)}{\big(|x-y|^2+\varepsilon^2\big)^{\frac{N+2s}{2}}}dxdy.
\end{align}
By \cite[Corollary $2$]{Brock-95}, we know that
\begin{equation}\label{Eq2.9}
\iint_{\R^N\times\R^N}\frac{u^t(x)u^t(y)}{\big(|x-y|^2+\varepsilon^2\big)^{\frac{N+2s}{2}}}dxdy\geq \iint_{\R^N\times\R^N}\frac{u(x)u(y)}{\big(|x-y|^2+\varepsilon^2\big)^{\frac{N+2s}{2}}}dxdy
\end{equation}
where we used that $\k_\varepsilon^t(z) = \k_\varepsilon(z)$ since it is radially decreasing. In view of \eqref{Eq2.8} and \eqref{Eq2.9} we get
\begin{align*}
J_\varepsilon[u^t]\leq      J_\varepsilon[u]\leq [u]^2_{H^s(\R^N)}.
\end{align*}
We conclude by using Fatou's lemma. In view of \eqref{pr}, to get \eqref{main-prop-0}, we simply need to check that $u^t\in \mathcal{H}^s_0(\O^t)$. For this, we let $u\in \mathcal{H}_0^s(\O)$, then there exists $u_n\in C^{0,1}_c(\O)$ such that $u_n\to u$ in $H^s(\R^N)$. By the second point of Lemma \ref{non-exp}, we have that  $u_n^t\in C^{0,1}(\R^N)$. We claim that $u_n^t$ is compactly supported in $\O^t$. Indeed, let $K=\supp(u)\subset\subset \O$, then by \cite[Lemma 1]{Brock-95}, we have for all $t\in [0,+\infty]$
\[
0<c_0\leq \dist(K,\partial\Om)\leq \dist(K^t,\partial\Om^t).
\]
By monotonicity --$(iii)$ of definition \eqref{defn:cont set}-- we easily check that 
$u^t \equiv 0$ in $(K^t)^c$. Since $K^t$ is compact, we conclude that $u_n^t\in C^{0,1}_c(\O^t)$.
It follows that
\begin{align*}
\|u_n^t\|^2_{\mathcal{H}^s_0(\O^t)}=[u_n^t]^2_{H^s(\R^N)}\leq [u_n]^2_{H^s(\R^N)}\leq C[u]^2_{H^s(\R^N)}
\end{align*}
where in the second inequality we have used \eqref{pr}. By compacity, there exists a subsequence $u_{n'}(t)$ and $w\in \mathcal{H}^s_0(\O^t)$ such that 
\[
u^t_{n'}\to w\quad\textrm{in $L^2(\O^t)$}\qquad \text{and}\qquad u^t_{n'}\to w\quad\textrm{weakly in $\mathcal{H}^s_0(\O^t)$}.
\]
To conclude, we need to check that $w=u^t$ and this follows from non-expanding property \eqref{non-exp}. Indeed, we have 
\begin{align*} 
\|w-u^t\|_{L^2(\R^N)}&\leq \|u_{n'}^t-w\|_{L^2(\R^N)}+\|u_{n'}^t-u^t\|_{L^2(\R^N)}\\
&\leq   \|u_{n'}^t-w\|_{L^2(\R^N)}+\|u_{n'}-u\|_{L^2(\R^N)}\to 0\qquad\textrm{as $n'\to\infty$}.
\end{align*}
Thus $u^t=w\in \mathcal{H}^s_0(\Om^t)$. This proves \eqref{main-prop-0}. The estimate \eqref{main-prop-1} follows immediately from Theorem \ref{thm-del} and the fact that $u^t\in \mathcal{H}^s_0(\O^t)$.

\end{proof}
\section{Characterization of solution domains}\label{sec-3}

Here, we relate solution domains $\O$  to  the overdetermined problem  \eqref{ov-det12} with critical sets of the functional 
\begin{equation}\label{eqenerJ}
\O\mapsto J_{s, p}(\O):=\lambda_{s, p}(\O)+C^2_0\Gamma^2(1+s)\, \textrm{Vol}(\O),
\end{equation}
where $\textrm{Vol}(\O)$ is the Lebesgue measure of $\O$ and $\Gamma$ is the classical gamma function.\\\\
To that aim, for every globally Lipschitz vector field $X\in C^{0,1}(\R^N,\R^N)$,  we let $\Phi_t:\R^N\to\R^N$ ($0\leq t<1$) be a family of homeomorphism such that $\Phi_0 = \id_{\R^N}$ and $\partial_t \Phi_t\Big|_{t=0} = X$. For a given shape functional $\O\mapsto J(\O)\in\R$, that is, a mapping that assigns to each measurable set a real number, we  let 
\[
dJ(\O)\cdot X= \lim_{t\to 0}\frac{J(\Phi_t(\O))-J(\O)}{t}
\]
whenever the limit exists. If the quantity above is finite, we say that $J(\O)$ is shape differentiable in the direction of $X$ and call $ dJ(\O)\cdot X$ the shape derivative in the direction of $X$.
\begin{lemma}\label{lem-3.1} Let $p\in [1,2]$. The functional $\Om\mapsto J_{s,p}(\O)$ defined in \eqref{eqenerJ} is shape differentiable in every direction $X\in C^{0,1}(\R^N,\R^N)$ and the shape derivative is given by:
\begin{equation}\label{shape-deriv-JO}
 dJ_{s,p}(\O)\cdot X=-\Gamma^2(1+s) \int_{\partial\O}\Bigl( (u/\delta^s)^2 - C_0^2\Bigl) X\cdot\nu\, d\sigma
\end{equation}
where $\nu$ denotes the outer unit normal to the boundary. 
\end{lemma}
\begin{proof}
It is easy to see that
\begin{equation*}\label{shap-vol}
 d\textrm{Vol}(\O)\cdot X= \int_{\partial\O} X\cdot \nu\,d\sigma.
\end{equation*}
In the other hand, we have (see e.g \cite[Corollary 1.2]{DFW-21}) 
\begin{equation*}\label{shape-deriv-lambda-sp}
d\lambda_{s, p}(\O)\cdot X = -\Gamma^2(1+s)\int_{\partial\O}(u/\delta^s)^2 X\cdot\nu\, d\sigma
\end{equation*}
Combining the above two inequalities gives \eqref{shape-deriv-JO}.  
\end{proof}

The following Proposition  characterizes solution domain of the overdetermined problem \eqref{ov-det12} and it is the main ingredient that our proof is based on.
\begin{proposition}\label{lem-4.2}
A bounded open set $\O$  of class $C^{1,1}$  is a solution of \eqref{ov-det12}
 if and only if 
\begin{equation}\label{main-observ2}
dJ_{s,p}(\O)\cdot X = 0\quad\textrm{for all $X\in C^{0,1}(\R^N,\R^N)$}. 
\end{equation}
\end{proposition}
\begin{proof}
If $\O$ is a solution of \eqref{ov-det12}, it directly follows from Lemma \ref{lem-4.2} that  $dJ_{s,p}(\O)\cdot X =0$.  Conversely  if \eqref{main-observ2}  holds then 
\begin{equation}\label{shap-44}
 \int_{\partial\O}\big((u/\delta^s)^2-C_0^2\big)X\cdot\nu d\sigma = 0
\end{equation}
for all $X\in C^{0,1}(\R^N,\R^N)$.
Fixed $h\in C^1(\partial\Om)$ and let $\widetilde{h}\in C^1(\R^N,\R)$ and $V\in C^{0,1}(\R^N,\R^N)$ be respectively a $C^1$-extension of $h$ and a $C^{0,1}$-extension of the unit normal. Plugging  $X = \widetilde{h}V\in C^{0,1}(\R^N,\R^N)$ into \eqref{shap-44} gives $\int_{\partial\O}\big((u/\delta^s)^2-C_0^2\big)h d\sigma = 0$ for all $h\in C^1(\partial\Om)$. By the fundamental theorem of calculus and since $u$ is nonnegative, we deduce that $u/\delta^s = C_0$ on $\partial\Om$.
\end{proof}
\section{ Proof of Theorem \ref{Theorem11}}\label{sec-4}
In this section, we prove our main result. We re-state the result for the reader convience.
\begin{theorem}Let $p\in [1,2]$. Then, every convex solution $\Om$ of \eqref{ov-det12} must be a ball.
\end{theorem}
\begin{proof}
The idea of the proof is to exhibit a globally Lipschitz vector field that violate  \eqref{main-observ2}.  Such a velocity field will be constructed by using the continuous Steiner symmetrization as in \cite{BH-02}. See also \cite{CH-99}.\\\\
Assume by contradiction that $\O$ is not a ball, then there a direction $e\in \mathbb{S}^1$ for which $\O$ is not symmetric. Without lost  of generality, we may assume that $e = e_N$. Since $\O$ is assumed to be convex, then for all $x'\in\O'$, there exists exactly two points $y_1(x'), y_2(x')\in \partial\O$ such that 
\[
\O(x') = (y_1(x'),y_2(x')).
\]
The set $\O$ can be thus characterized by 
\begin{equation}\label{charac-Om}
    \O = \Big\{ (x',y): x'\in \O'\quad y_1(x')\leq y\leq y_2(x')\Big\}.
\end{equation}
The continuous Steiner symmetrization of $\O$ is then given by 
\begin{equation}\label{cont-St-sym}
\O^t = \Big\{ (x',y): x'\in \O'\quad y\in \big(y_1^t(x'),y_2^t(x')\big)\Big\},   
\end{equation}
where for $0\leq t\leq \infty$ we set
\begin{align}
y_1^t(x') &= \frac{1}{2}\Big(y_1(x') -y_2(x') +e^{-t}\big(y_1(x') +y_2(x') \big)\Big),\label{y-1t}\\
y_2^t(x') &= \frac{1}{2}\Big(y_1(x') -y_2(x') +e^{-t}\big(y_1(x') +y_2(x') \big)\Big)\label{y-2t}.
\end{align} 
Recalling the characterization of $\Omega$ in \eqref{charac-Om}, we define the vector field 
\[
V: \R^N\to \R^N, \quad x\mapsto V(x',y)=
\left\{\begin{aligned}
\big( 0, -(1/2)\big(y_1(x')+y_2(x')\big)\big)&\quad\text{if}\quad\textrm{$(x',y)\in \overline{\O},$}\\
0&\quad \text{else}.
\end{aligned}
\right. 
\]
Since $\O$ is convex, then $V\in C^{0,1}(\R^N,\R^N) $. Indeed, for each $j\in \{1,2\}$, the function 
\[
y_j: \R^{N-1}\to\R, \quad x'\mapsto y_j(x')=
\left\{\begin{aligned}
y_j(x')&\quad\text{if}\quad x'\in \overline{\O'}\\
& 0\quad \text{if}\quad\R^{N-1}\setminus \overline{\O'}
\end{aligned}
\right. 
\]
are Lipschitz continuous. To see this, we let $\{x_l\}_{l\in I}$ be a finite collection of points in $\O'$, and let $
\{B(x_l 
,\mu_0)\}_{l\in I}$ be a finite family of open balls, each of radius $\mu_0>0$ such that 
$$\overline{\O'}\subset \bigcup_{l\in I} B(x_j,\mu_0).
$$
Since $y_j\in L^\infty(\R^{N-1})$, we may assume without loose of generality that $x'_1$ and $x'_2$ belong both to one of the balls $B(x_l, \mu_0)$ with $l\in I$,  and thus $|x'_1-x'_2|\leq 2\mu_0$. Because $\O$ is convex, its boundary is locally the graph of a Lipschitz continuous function. That is, for a fixed $j\in \{1,2\}$ and for every $x_l$, there exists a cylinder $K^j_l = B'\times (-a,+a)$ where $B'\in \R^{N-1}$ is an open ball centered at $y_j(x_l)\in \partial\O$, and $a>0$ a positive real number and a Lipschitz function $\psi^j_l: B'\to (-a,+a)$ such that $\psi^j_l(x_l) = 0$, $ \partial\O \cap K^j_l= \Big\{ (x',\psi^j_l(x')); x'\in B'\Big\}$ and $\O\cap K^j_l = \Big\{(x',x_N)\in K;\; x_N>\psi^j_l(x')\Big\}$.
Up to reducing $\mu_0>0$ if necessary, we may assume that $y_j(x'_1),y_j(x'_2)\in \partial\Om\cap K^j_l$. Hence, 
\[y_j(x'_1) = \psi^j_l(x'_1) \quad\text{and} \quad y_j(x'_2) = \psi^j_l(x'_2).\]
Since $\psi^j_l$ is Lipschitz, we have 
\[
|y_j(x'_1)-y_j(x'_2)| = |\psi^j_l(x'_1)-\psi^j_l(x'_2)|\leq C^j_l|x'_1-x'_2|.
\] 
Set \[C_j = \max_{l\in I} C^j_l.\]
Then for every couple $x'_1,x'_2\in \bigcup_{l\in I} B(x_l,\mu_0)\supset \overline{\O'}$ with $|x'_1-x'_2|\leq 2\mu_0$, we get $|y_j(x'_1)-y_j(x'_2)|\leq C^j|x'_1-x'_2|$. That is, $y_j$ is globally Lipschitz continuous which in turn yields the Lipschitz continuity of $V$.\\\\
Next, for $t\in [0,\infty)$, we define 
\[
\Phi_t:\R^N\to\R^N, \quad x\mapsto\Phi_t(x) =  x+(1-e^{-t})V(x).
\]
For $0\leq t<1$ small enough, $\Phi_t$ is a homeomorphism and we have $\O^t = \Phi_t(\O)$. Next, we let $u_t\in \mathcal{H}^s_0\big(\Phi_t(\O)\big)$ be the unique weak solution of  
\begin{equation*}
\left\{\begin{aligned}
\Ds u_t&= \lambda_{s, p}(\O) u_t^{p-1}\quad\text{in}\quad\Phi_t(\Om) \\
u_t &= 0\quad \text{in}\quad\R^N\setminus \Phi_t(\Om).
\end{aligned}
\right. 
\end{equation*}
In what follows we shorten the notation and write  $\lambda(\O)$ in place of $\lambda_{s, p}(\O)$. 
Then, we have
\begin{align}\label{charac-J}
\lambda(\Phi_t(\O)) &=  \inf_{\mathcal{H}^s_0\big(\Phi_t(\O)\big)}\Big\{[u]^2_{H^s(\R^N)},\, \|u\|^2_{L^p(\R^N)}=1\Big\}\nonumber\\
&= \inf_{\mathcal{H}^s_0(\O^t)}\Big\{[u]^2_{H^s(\R^N)},\, \|u\|^2_{L^p(\R^N)}=1\Big\}.
\end{align}
By Proposition \ref{prop-1}, we know that $u^t$ --the continuous Steiner symmetrization of $u$-- which is a priori different from $u_t$, is admissible in \eqref{charac-J}. Consequently, and since the Steiner symmetrization preserves the $L^p$-norm, we deduce from \eqref{main-prop-1} that there exists $\gamma_0  = \gamma_0(N,s,u)>0$ and $t_0 = t_0(N,s, u)$ such that for $0\leq t<t_0$, there holds:
\[
\lambda (\Phi_t(\O)) -\lambda(\O) \leq \|u^t\|^2_{\mathcal{H}^s_0(\O^t)}-\|u\|^2_{\mathcal{H}^s_0(\O)}\leq -\gamma_0 t.
\]
In other words
\[
\frac{\lambda (\Phi_t(\O)) -\lambda(\O)}{t}\leq -\gamma_0<0
\]
Taking the limit (which exists in view of Lemma \ref{lem-3.1} ) we get 
\begin{equation}\label{dJ.V}
d\lambda (\O)\cdot V = \lim_{t\to 0}\frac{\lambda (\Phi_t(\O)) -\lambda (\O)}{t}\leq -\gamma_0<0
\end{equation}
In the other hand, by $(ii)$ of definition \eqref{defn:cont set}, we know that the continuous Steiner symmetrization is volume preserving, thus
\begin{equation}\label{dVol.V}
d\textrm{Vol}(\O)\cdot V  = 0.
\end{equation}
In view of \eqref{dJ.V} and \eqref{dVol.V}, we get $dJ_{s,p}(\O)\cdot V <0$ which contradicts \eqref{main-observ2}. The proof is finish.
\end{proof}


\end{document}